\documentclass[12pt]{iopart}
\usepackage{graphicx}

\expandafter\let\csname equation*\endcsname\relax
\expandafter\let\csname endequation*\endcsname\relax

\usepackage{amsthm}
\usepackage{amsmath}
\usepackage{iopams}

\newcounter{oftheorem}[section]
\newenvironment{mytheorem}[1]%
{\begin{trivlist}
     \renewcommand{\theoftheorem}{#1~\thesection.\arabic{oftheorem}}
     \refstepcounter{oftheorem}
     \item[\hspace{\labelsep}\bf\thesection.\arabic{oftheorem} #1.]}%
{\end{trivlist}}
\newenvironment{definition}{\begin{mytheorem}{Definition}}{\end{mytheorem}}

\newenvironment{proposition}{\begin{mytheorem}{Proposition}\it}{\end{mytheorem}}

\newenvironment{lemma}{\begin{mytheorem}{Lemma}}{\end{mytheorem}}

\providecommand*{\threeD}{\textsc{3d}}

\newcommand{\FOURVEC}[4]{\left(\begin{array}{l}#1\\#2\\#3\\#4\end{array}\right)}
\newcommand{\TWOVEC}[2]{\left(\begin{array}{l}#1\\#2\end{array}\right)}

\newcommand{\numu}{n_{\text{u}}}
\newcommand{\nums}{n_{\text{s}}}

\newcommand{\uu}{u_{\text{u}}}
\newcommand{\us}{u_{\text{s}}}
\newcommand{\vu}{v_{\text{u}}}
\newcommand{\vs}{v_{\text{s}}}

\newcommand{\Su}{S^{\text{u}}}
\newcommand{\Pu}{P^{\text{u}}}

\newcommand{\Ss}{S^{\text{s}}}
\newcommand{\Ps}{P^{\text{s}}}

\newcommand{\Wu}{W^{\text{u}}}
\newcommand{\Ws}{W^{\text{s}}}

\newcommand{\tu}{t_{\text{u}}}
\newcommand{\ts}{t_{\text{s}}}

\newcommand{\Eu}{E^{\text{u}}}
\newcommand{\Es}{E^{\text{s}}}
\newcommand{\Esu}{E^{\text{s,u}}}

\newcommand{\fu}{f^{\text{u}}}
\newcommand{\fs}{f^{\text{s}}}

\newcommand{\xh}{x^{\text{h}}}
\newcommand{\yh}{y^{\text{h}}}

\newcommand{\xe}{x^{\text{e}}}
\newcommand{\ye}{y^{\text{e}}}

\newcommand{\deltac}{\delta_{\text{c}}}

\newcommand{\lambdau}{\lambda_{\text{u}}}
\newcommand{\lambdas}{\lambda_{\text{s}}}

\newcommand{\ui}{\text{i}}
\newcommand{\ud}{\text{d}}

\usepackage{hyperref}

\newcommand{\movierefall}{For a rotating view see
\href{http://www.comp-phys.tu-dresden.de/supp/}{http://www.comp-phys.tu-dresden.de/supp/}.}

\begin{document}
\title{Homoclinic Points of 2-D and 4-D Maps via the Parametrization Method}

\author{Stavros Anastassiou$^1$, Tassos Bountis$^{1,2}$ and Arnd B\"acker$^{3,4}$}
\address{$^1$Center of Research and Applications of Nonlinear Systems (CRANS),\\
 University of Patras, Department of Mathematics,\\ GR-26500 Rion, Greece}
\address{$^2$Department of Mathematics,School of Science and Technology,\\
          Nazarbayev University, Kabanbay-batyr, 53, 010000 Astana, Kazakhstan}
\address{$^3$Technische Universit\"{a}t Dresden,\\
          Institut f\"{u}r Theoretische Physik and Center for Dynamics, 01062 Dresden, Germany}
\address{$^4$Max-Planck-Institut f\"{u}r Physik komplexer Systeme, N\"{o}thnitzer Strasse 38,\\
 01187 Dresden, Germany}

 \ead{SAnastassiou@gmail.com}

\begin{abstract}

An interesting problem in solid state physics is to compute discrete breather solutions in
$\mathcal{N}$ coupled 1--dimensional Hamiltonian particle chains and investigate the richness
of their interactions. One way to do this is to compute the homoclinic intersections of
invariant manifolds of a saddle point located at the origin of a class of $2\mathcal{N}$--dimensional
invertible maps. In this paper we apply the parametrization method to express these manifolds
analytically as series expansions and compute their intersections numerically to high precision.
We first carry out this procedure for a 2--dimensional (2--D)  family of generalized H\'{e}non maps
($\mathcal{N}=1$), prove the existence of a hyperbolic set in the non-dissipative case and show that
it is directly connected to the existence of a homoclinic orbit at the origin. Introducing dissipation
we demonstrate that a homoclinic tangency occurs beyond which the homoclinic intersection disappears.
Proceeding to $\mathcal{N}=2$, we use the same approach to accurately determine the homoclinic intersections of the
invariant manifolds of a saddle point at the origin of a 4--D map consisting of two coupled 2--D cubic H\'{e}non maps.
For small values of the coupling we determine the homoclinic intersection,
which ceases to exist once a certain amount of dissipation is present.
We discuss an application of our results to the study of discrete breathers in two linearly coupled
1--dimensional particle chains with nearest--neighbor interactions and a Klein--Gordon on site potential.
\end{abstract}

\vspace{2pc}

\maketitle

\section{Introduction}
\label{introduction}

An important topic in the study of the dynamics of 1--dimensional lattices (or chains) of nonlinearly
interacting particles is their ability to support under rather general conditions a very interesting type
of localized oscillations called discrete breathers (see \cite[Chpt.~7]{BouSko2012}, \cite{flach}). These
solutions simply execute periodic motion and involve one or more central particles that carry most of the
energy, while all others in their immediate vicinity have amplitudes that vanish exponentially as the index of
the particle $n$ goes to $+\infty$ or $-\infty$. Take for example the so--called Klein-Gordon
system of ordinary differential equations written in the form
\begin{equation}
\ddot u_n=-V^\prime\left(u_n\right)+\alpha\left(u_{n+1}-2u_n+u_{n-1}\right), \qquad
V(x)=\frac{1}{2}Kx^2+\frac{1}{4}x^4,\label{eq4}
\end{equation}
where $u_n$ for $-\infty<n<\infty$ is the amplitude of the $n$-th particle, $\alpha>0$ is a parameter indicating the strength of coupling between nearest neighbors,
and $V(x)$ is the on-site potential with primes denoting differentiation with respect to the argument of $V(x)$.

To construct such a discrete breather solution one may insert a Fourier series
\begin{equation}
\label{eq5}
u_n(t)=\sum_{k=-\infty}^{\infty}A_n(k)\exp(\ui k\omega_bt)
\end{equation}
in the equations of motion (\ref{eq4}), where $\omega_b$ is the frequency of the breather.
Setting the terms proportional to the same exponential equal to
zero one obtains the system of equations
\begin{equation}
\label{infdimmap}
-k^2\omega_b^2A_n(k)= \alpha\left(A_{n+1}(k)-2A_n(k)+A_{n-1}(k)\right) - KA_n(k)-\sum_{k_1,k_2,k_3}A_n(k_1)A_n(k_2)A_n(k_3),
\end{equation}
where $k_1+k_2+k_3=k, \forall k,\ n\in\mathbb{Z}$.
This equation defines an infinite-dimensional mapping in the space of
Fourier coefficients $A_n(k)$. Time-periodicity is ensured by the Fourier basis functions
$\exp\left(\ui k\omega_bt\right)$, while spatial localization
requires that $A_n(k)\rightarrow 0$ exponentially as $|n|\rightarrow \infty$.

If we want to construct a breather solution one can start from its lowest order approximation
by substituting $u_n(t)=2A_n(1)\cos(\omega_bt)$ in (\ref{infdimmap}) and obtain a 2--dimensional map for the largest coefficient
$A_n(1)=A_n(-1)=A_n$ as follows
\begin{equation}
-\omega_b^2A_n= \alpha(A_{n+1}-2A_n+A_{n-1})-KA_n-3A_n^3,\,\,\,\
\end{equation}
which may be written in the form
\begin{equation}
 A_{n+1}=-A_{n-1}+\frac{1}{\alpha}\left(2+K-\omega_b^2\right)A_n + \frac{3}{\alpha}A_n^3\,\,\,\label{2dimAnmap}.
\end{equation}
If we now define $x_n=A_n$ and $y_n=A_{n-1}$ the mapping \eqref{2dimAnmap} takes the 2--dimensional (2--D) form
\begin{subequations}\label{2dimxnynmap}
\begin{align}
     x_{n+1} &= -y_n+Cx_n +\frac{3}{\alpha}x_n^3, \\
     y_{n+1} &= x_n
\end{align}
\end{subequations}
where $C=(2+K-\omega_b^2)/\alpha$. This mapping
is area-preserving (and also symplectic) since the determinant of its Jacobian is unity for all $x_n$ and $y_n$.
The above approach has proved quite useful in the past and has led to a wide variety of interesting results
concerning the computation and dynamics of breathers and multibreathers in a large family of 1--dimensional
Hamiltonian lattices \cite{BCKRBW,BBJ,BBV}. To our knowledge, such a study has not yet been carried out
for coupled Hamiltonian chains of this type.

Breather solutions of the original chain (\ref{eq4}) by definition must have large values of $|A_n|$ for small $n=0,\pm1, \pm2,...$, while $A_n\rightarrow 0$ as $n\rightarrow \pm\infty$. This implies that these amplitudes can be identified as homoclinic orbits lying at the intersections of stable and unstable manifolds of the origin, which if hyperbolic is necessarily  a saddle point of the map (\ref{2dimxnynmap}). For any homoclinic intersection point, its orbit under forward (backward) iteration along the stable (unstable) manifold asymptotically converges to the origin. Thus the
requirement that $A_n\rightarrow 0$ as $n\rightarrow \pm\infty$ is fulfilled at the outset.
To locate such homoclinic orbits accurately, one must be able to write down precise expressions for the curves representing these manifolds and compute their points of intersection.

The purpose of this paper is twofold: First, we develop and apply the parametrization method to compute intersections of invariant stable and unstable manifolds of a certain class of 2--D and 4--D invertible maps. These are the types of maps used to obtain the largest coefficients $A_n$ and $B_n$ in the Fourier expansion of the $n$--th particle of two coupled chains. So far, such 2--D maps have been successfully used to obtain such $A_n$ approximation for 1--D Hamiltonian particle chains with Klein-Gordon on-site potential (see \cite{BCKRBW,BBJ,BBV,Bergamin}). Second, we locate the coordinates of the homoclinic points of such mappings, along with the critical value of the dissipation parameter for which homoclinic intersections no longer exist. The accuracy of our computations of homoclinic points of 4--D maps is very encouraging, as an application of these techniques to coupled 1D particle chains is now possible.

The success of this approach in studying the existence of breathers relies on the fact that breather computation is achieved by using rapidly convergent Newton schemes to compute the breathers as simple periodic orbits of a Hamiltonian system of differential equations. As is well--known, this crucially relies on having an accurate first approximation. This is why the knowledge of the first Fourier coefficients from the homoclinic orbits of the corresponding maps is so useful. Indeed, higher Fourier coefficients are not needed, since they are obtained through the convergence of the Newton method to the exact breather solution.

The paper is structured as follows: First, in section 2, we introduce the 2--D mapping of interest here and prove its hyperbolic behavior for the symplectic case. In section \ref{parametrization} we explain the idea of the parametrization method and how we use it to obtain power series for the manifold equations. In section 4 we solve these equations numerically and obtain the homoclinic intersections for a cubic map of the form (\ref{2dimxnynmap}). Finally, in section \ref{4-d case}, we apply our approach to a 4--D map with cubic nonlinearities and compute manifold intersections for various parameter values, commenting also  on the accuracy limitations encountered by our numerical algorithms in this process. We close with our concluding remarks in section \ref{conclusions}.
\section{Hyperbolicity in a Family of 2--Dimensional H\'{e}non Maps}
\label{hyperbolicity}

\subsection{Generalized H\'{e}non maps}

Nearly 45 years ago the French astronomer Michel H\'{e}non introduced a 2--D mapping of the plane onto itself with the simple form \cite{henon}
\begin{equation}
\label{H2dmap}
 h:\mathbb{R}^2\rightarrow \mathbb{R}^2,\ h(x,y)=(1+y-ax^2,bx),
\end{equation}
which exhibits very interesting phenomena related to chaos, bifurcations and strange attractors for different
values of its parameters $a>0$ and $|b|\leq1$. In fact, the occurrence of some  of the most important properties of (\ref{H2dmap})
have been related to the transition from simple dynamics
to hyperbolicity (see e.g.~\cite{devnit, russo, SteDulMei1999}). This implies that there are
dense sets of chaotic orbits lying at the intersections of invariant manifolds of
unstable (saddle) fixed points and periodic orbits \cite{wiggins1}.

Nowadays, however, it is more common to consider instead of (\ref{H2dmap}) its conjugate expression
\begin{equation}
\label{TH2map}
 h_s:\mathbb{R}^2\rightarrow \mathbb{R}^2,\ h(x,y)=(y,-bx+a-y^2),
\end{equation}
which is more convenient to generalize both in form as well as number of dimensions, as described for
example in \cite{lomelimeiss} and \cite{gonmeov}. Thus, let us consider the family of generalized H\'{e}non maps
of the plane onto itself defined by
\begin{equation}
\label{GH2map}
 H:\mathbb{R}^2\rightarrow \mathbb{R}^2,\ H(x,y)=(y,-\delta x+p(y)),
\end{equation}
where $p(y)$ is a univariate polynomial.
This class is of particular importance as any
polynomial mapping of the plane having a polynomial inverse
(i.e.\ any member of the 2--D affine Cremona group)
is either a composition of mappings of the form (\ref{GH2map}), or possesses
trivial dynamics \cite{milnor}.

The main properties of the generalized H\'{e}non family (\ref{GH2map}) are the following: First,
the inverse of $H(x,y)$ is explicitly given as $H^{-1}(x,y)=(-\frac{1}{\delta}y+\frac{1}{\delta}p(x),x)$.
Moreover, if the polynomial $p(y)$ is odd, the mapping $H(x,y)$ is symmetric under the transformation
$\sigma (x,y)=(-x,-y)$, which implies $H\circ \sigma =\sigma \circ H(x,y)$.
For $\delta =1$ the mapping $H(x,y)$ is a symplectomorphism (or symplectic map), as
it preserves the natural symplectic form of the plane,
$\ud x\wedge \ud y$. In addition, $H$ is also differentially
conjugate to its inverse, since, if we define $\rho (x,y)=(y,x)$, it follows that $H\circ \rho = \rho \circ H^{-1}$ holds.
As a consequence, in the symplectic case, invariant sets of $H(x,y)$ are related to invariant
sets of $H^{-1}$ by the transformation $\rho$.

Generalized H\'{e}non maps have attracted a lot of attention in the literature.
For example, in \cite{dullinmeiss}
the dynamics of $H(x,y)$ is studied, where $p(y)$ is a polynomial of third degree, while in \cite{gonkuzmei}
bifurcations of homoclinic tangencies are considered, involving intersections of invariant manifolds,
for a mapping that has many similarities with (\ref{GH2map}) above. Furthermore, we note that in \cite{zhang}
sufficient conditions are given for hyperbolicity in generalized H\'{e}non maps for arbitrary polynomials $p(y)$.

In the present work we study a class of maps (\ref{GH2map}) that corresponds to the choice of a cubic polynomial $p(y)$ and generalize our results to the case of 4--D maps, as discussed in section \ref{introduction} in connection with applications to discrete breathers in systems of 1--dimensional Hamiltonian particle chains \cite{BouSko2012,BCKRBW,BBJ,BBV}.
\\

\subsection{The 2--dimensional cubic map}

Let us focus now on the dynamics of the recurrence relation:
\begin{equation}
A_{n+1}=cA_n-\delta A_{n-1}+3A_n^3.
\end{equation}
which follows from (\ref{2dimAnmap}) by a simple rescaling of the
Fourier coefficients, while a parameter $0\leq \delta \leq 1$ has been introduced to account for dissipation effects.
Setting $A_{n-1}=x,\ A_n=y$, we define the following cubic map of
the plane $x,y$ onto itself
\begin{equation} \label{eq:our-cubic-map}
f:\mathbb{R}^2\rightarrow \mathbb{R}^2,\ (x,y)\mapsto (y,-\delta x+cy+3y^3),
\end{equation}
which corresponds to the generalized H\'{e}non map \eqref{GH2map}
for $p(y) = cy+3y^3$.
Its inverse is given by
\begin{equation}
f^{-1}(x, y) = \left(\frac{c}{\delta}x
               -\frac{1}{\delta}y+\frac{3}{\delta} x^3,
               x\right).
\end{equation}
For $\delta=1$ the mapping $f$ is a symplectomorphism
and is invariant under the symmetry $\sigma(x,y) = (-x,-y)$, i.e.\
the relation $f\circ \sigma = \sigma \circ f$ holds.

\subsection{Existence of a hyperbolic set in the symplectic case $\delta =1$}

The generalized H\'{e}non map is known to have rich dynamics. In fact, in \cite{zhang}
the existence of hyperbolic sets for such maps has been proven for a wide range of parameter
values.
\begin{definition}
 Let $g:M\rightarrow M$ be a diffeomorphism and $\Lambda$ a compact subset of $M$, invariant
 under this diffeomorphism. $\Lambda$ is said to be a hyperbolic set for $g$
 if $\forall x\in \Lambda$ the following is true:
 \begin{enumerate}
  \item {$T_x M=\Es_x\oplus \Eu_x$,}
  \item {$d_xg(\Esu_x)=\Esu_{g(x)}$,}
  \item {$\| d_xg|_{\Es_x}\| <\lambda,\  \| d_xg^{-1}|_{\Eu_x}\| <\mu ^{-1}$, for $0<\lambda <1<\mu$.}
 \end{enumerate}
\end{definition}

These conditions imply that the tangent space of $M$ at every point of $\Lambda$ is the direct sum of two
subspaces $\Es_x$ and $\Eu_x$ which are invariant under the action of the differential of $g$.
Moreover on these subspaces the differential acts as a contraction and dilation, respectively.

In what follows we will be interested in the diffeomorphism $f$ defined by equation~\eqref{eq:our-cubic-map}
for the particular choice of $c=-\frac{5}{2}$ which yields a saddle point at the origin and will be kept fixed in the remainder of this paper. This choice is pictorially convenient, since $c$ values in that range produce large scale manifolds that are clearly visible in the figures. Let us now demonstrate the existence of a hyperbolic set for this system:
\begin{proposition}
Let $f$ be the diffeomorphism~\eqref{eq:our-cubic-map}, with $c=-\frac{5}{2}$ and $\delta =1$. Then, there
exists a hyperbolic set $\Lambda$ on which $f$ is topologically conjugate to the two--sided shift on
three symbols. Moreover, $\Lambda $ consists of all those points whose orbits are bounded under successive iterations of~~\eqref{eq:our-cubic-map}.
\end{proposition}
\begin{proof}
The proof utilizes theorem 3.3 of \cite{zhang}. However, since it is not directly
applicable to $f$, we first  perform a coordinate change by defining
$h_1,\ h_2:\mathbb{R}^2\rightarrow \mathbb{R}^2$
as $h_1(x,y)=(\frac{1}{a}x,y)$ and $h_2(x,y)=(x,\frac{1}{a}y)$.
Then, $f$ is left--right equivalent to
$\hat{f}:\mathbb{R}^2\rightarrow \mathbb{R}^2$,
defined by $\hat{f}(x,y)=(y,-\delta x+3a^3y^3-\frac{5}{2}ay)$,
that is, $h_1\circ f=\hat{f}\circ h_2$.
Note that this is nothing more than a coordinate transformation of
both the domain of definition and the image of $f$, so, qualitatively,
the behavior of $f$ remains unchanged.

Let us now set $a=5$ in this example. Following the notation used in \cite{zhang}, we notice that $\hat{f}$ is of the form $\hat{f}(x,y)=(y,-\delta x+p(y))$,  where $p(y)=375(y^3-y/30)=375g(y)$, with $g(y)=y^3-y/30$. Determining the roots of $g(y)$ as $\alpha _1=-1/\sqrt{30},\ \alpha _2=0,\ \alpha _3=1/\sqrt{30}$, we consider the intervals $V_1=(-1.03,-0.149),\ V_2=(-0.05,0.05),\ V_3=(0.149,1.03)$.
Having thus chosen the neighborhoods of the map that are of interest, it is easy to verify that all conditions required for the application of Theorem 3.3 in \cite{zhang}
are satisfied, therefore the proposition holds.
\end{proof}

\begin{figure}[b]
\begin{center}
\includegraphics[scale=1.2]{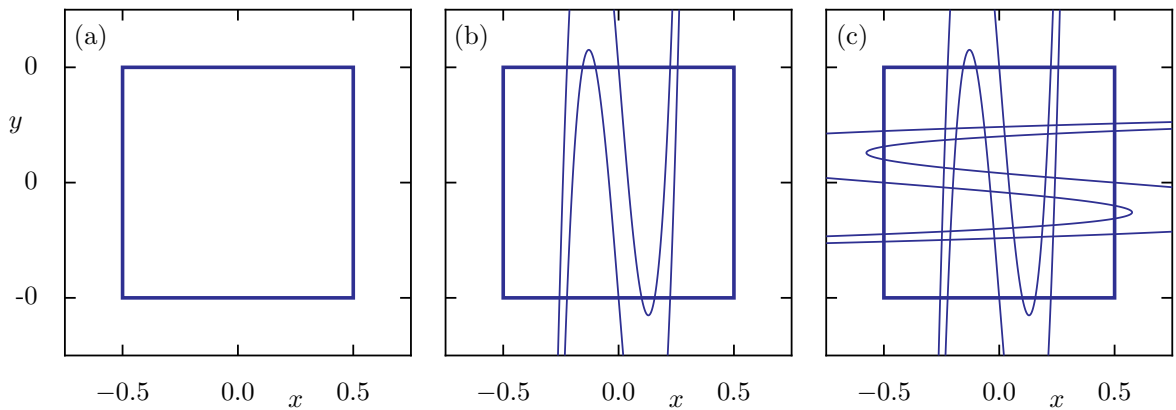}
\end{center}
\caption{\label{threesymbols}%
(a) The square $Q=[-1/2, 1/2]\times [-1/2,1/2]$. (b) The area within the three vertical strips represents all points in the set $Q\cap \hat{f}(Q)$ generated after one forward iteration of all points of $Q$. (c) The points of $Q$ that are still inside $Q$ after one forward and one backward iteration of the map lie in the intersection of the vertical strips of (b) and the horizontal strips shown here. This set, $\hat{f}^{-1}(Q)\cap Q\cap \hat{f}(Q)$, represents the first step in the construction of the hyperbolic set of $\hat{f}$.
}
\end{figure}

Figure \ref{threesymbols} shows the first few steps of the construction of the hyperbolic set for
the above map $\hat{f}$.

Hyperbolic sets persist for small perturbations of the mapping, thus we expect to find hyperbolic
behavior of $f$ for other values of $\delta$ as well. In later sections we show
that the hyperbolic set established by the above proposition owes its existence to a
transverse homoclinic point of $f$. Then we shall follow the intersections
of the invariant manifolds of $f$ emanating from the origin to locate the critical value
$\deltac<1$ at which these manifolds are tangent to each other. This means that for $\delta<\deltac$ homoclinic orbits no longer exist and the chaotic behavior of the map about the origin disappears \cite{papageo}.
To determine the invariant manifolds and compute
their intersections, we make use of
the so--called parametrization method, which we
now briefly outline.

\section{Overview of the Parametrization Method}
\label{parametrization}

Let $f:\mathbb{R}^n\rightarrow \mathbb{R}^n$ denote a $C^\infty $ diffeomorphism, having
a hyperbolic fixed point at $p\in \mathbb{R}^n$. Assume that $\Es$ is the eigenspace
of $d_pf:\mathbb{R}^n\rightarrow \mathbb{R}^n$, corresponding to eigenvalues that have norm less than one.
The stable manifold theorem \cite{hps} asserts  that the stable set of $p$,
$\Ws(p):=\{x\in \mathbb{R}^n \;|\; \lim_{n\rightarrow +\infty}f^n(x)=p\}$, is a smooth immersed submanifold of
$\mathbb{R}^n$, tangent to $\Es$ at $p$ (an analogous statement holds for the unstable manifold of $p$).

Knowing how these two manifolds evolve in the phase space of the map offers crucial insight in
the dynamics of the diffeomorphism $f$. Consequently, a number of methods have been developed to compute and
visualize these manifolds. see \cite{survey} for a survey of these methods and more recently \cite{goodman} for the 2--D case. Here we choose the parametrization method for the computation of such manifolds, which presents a number of advantages (see \cite{par3} for more details).

Let us now demonstrate how to compute stable (and unstable) invariant manifolds using the parametrization method.
This method offers a simple and straightforward procedure for computing invariant manifolds of vector fields and
diffeomorphisms as explained in detail in  \cite{par1, par2, par3}. In fact, its applications go far beyond the
topics we examine in this paper (see \cite{partori, jdm, jdm2, HarCanFigLuqMon2016,Mireles-Lomeli,Cap-James} and references therein,
for applications in a wide variety of problems).

The parametrization method builds on the following fact: Under the assumptions imposed  on the diffeomorphism $f$, there exists a
$C^\infty$ injective immersion $S:\Es\rightarrow \mathbb{R}^n$, such that:\\
(a) $S(p)=p$,\\
(b) the derivative of $S$ at $p$ is the inclusion map $\Es\hookrightarrow \mathbb{R}^n$, and\\
(c) $f\circ S=S\circ \fs$, where $\fs$ stands for the restriction of $d_pf$ to $\Es$.\\
This version of the stable manifold theorem may be found in \cite{smale}, and the
immersion $S$ is to be thought of as the parametrization of $\Ws(p)$, considered as
a submanifold of $\mathbb{R}^n$.

Central to the implementation of the parametrization method is the equation
\begin{equation} \label{eq:defining-equation}
  f\circ S=S\circ \fs,
\end{equation}
which from now on we shall call the {\it defining equation} of the stable manifold.
Here $f$ stands for the diffeomorphism of interest, while $\fs$ represents its linear part determined
by the stable eigenvalues of the fixed point. $S$ denotes the parametrization of
the invariant manifold that we wish to compute.

To proceed with this computation, we first need to expand the map $S$ as a
power series. Inserting this power series into the defining equation, we shall
arrive at relations giving the coefficients of terms of degree $n$ as a function of
coefficients of terms of lower degree. What is especially convenient here is the fact that
these relations are linear and are thus easily solved, provided all coefficients of
lower order terms are given (for a specific application see section 4.1).

Making use of the facts (a) and (b) above, one immediately finds that the constant terms of
the power series are none other than the coordinates of the fixed point $p$, while the
coefficients of the linear terms are given by the eigenvectors of the stable
eigenvalues. Thus, solving a set of linear equations as explained above, we may
compute one by one the coefficients of our power series up to arbitrary (but
finite) order. This series does indeed converge under mild assumptions on $f$,
as explained in detail in \cite{par3}.

In a completely analogous way the unstable manifold of $p$ can also be
computed. To accomplish this one simply has to replace $\fs$ in the defining equation \eqref{eq:defining-equation} by $\fu$, that is, the restriction of $d_pf$ onto the
unstable subspace $\Eu$ of $p$. From here on, the procedure explained above is
employed in precisely the same way to provide us with a parametrizion of
the unstable manifold of $p$.

Let us now apply this technique to construct the invariant manifolds of the cubic 2--D diffeomorphism of interest here.

\section{Invariant manifolds for the cubic diffeomorphism of the plane}
\label{sec:2-dim-map}

Returning to the cubic mapping given in equation~\eqref{eq:our-cubic-map}, let us note that it
possesses three fixed points:
The first one is the origin and exists for all parameter values,
while there are also two symmetric ones, with coordinates
$\left(\pm\frac{\sqrt{1-c+\delta}}{\sqrt{3}},
       \pm\frac{\sqrt{1-c+\delta}}{\sqrt{3}}\right)$.
Focusing at the $(0,0)$ fixed point, we first determine the eigenvalues of the linearized map
\begin{equation}
  \frac{1}{2}\left(c-\sqrt{c^2-4\delta}\right),\
   \frac{1}{2}\left(c+\sqrt{c^2-4\delta}\right),
\end{equation}
and proceed to study its invariant manifolds, beginning with the $\delta =1$
case and continuing with $\delta <1$, where $f$ becomes dissipative. Our
purpose is to locate the homoclinic points which are part of the
hyperbolic set of $f$. In fact, it suffices to locate the primary one at which the manifolds first meet, since it ``generates'', under repeated application
of $f$ and $f^{-1}$, all other points of the associated homoclinic
orbit.

\begin{definition}\cite{wiggins2}
  Let $p\in \Ws(0,0)\cap \Wu(0,0)$, and denote by $S[(0,0),p]$ the segment of $\Ws(0,0)$
  with endpoints $(0,0)$ and $p$ and by $U[(0,0),p]$ the segment of $\Wu(0,0)$ with
  endpoints $(0,0)$ and $p$. The point $p$ is called a primary (homoclinic) intersection (point)
  if $S[(0,0),p]$ intersects $U[(0,0),p]$ only at the points $p$ and $(0,0)$.
\end{definition}

As before, to fix ideas we set $c=-5/2$ and proceed in what follows with the computation of the primary intersection points for $f$, using the parametrization method.

\subsection{The symplectic case $\delta =1$}

Since $c=-5/2$ and $\delta=1$, the eigenvalues of the origin are
$\lambdau = -2$ and $\lambdas = -1/2$
with normalized eigenvectors $(-1/\sqrt{5}, 2/\sqrt{5})$ and $(-2/\sqrt{5}, 1/\sqrt{5})$.
The origin is therefore a saddle, with a 1--dimensional stable and a 1--dimensional
unstable manifold, which we now proceed to determine.

Let $\Su:\mathbb{R}\rightarrow \mathbb{R}^2$ be the parametrization of the unstable manifold emanating from the
origin expressed by the expansion
\begin{equation}
\Su(t)=\left(\sum_{n=0}^{+\infty} a_nt^n, \sum_{n=0}^{+\infty}b_nt^n\right).
\end{equation}
The defining equation of this manifold becomes
\begin{equation} \label{eq:defining-equation-stable-manifold}
  f(\Su(t))=\Su(\lambdau t)
\end{equation}
so that in the case of our 2--D map we obtain
\begin{equation}
\TWOVEC{\displaystyle \sum_{n=0}^{+\infty}b_nt^n}
       {\displaystyle -\delta\sum_{n=0}^{+\infty}a_nt^n+c\sum_{n=0}^{+\infty}b_nt^n
                      +3\left(\sum_{n=0}^{+\infty}b_nt^n\right)^3}
= \TWOVEC{\displaystyle \sum_{n=0}^{+\infty}a_{n}\lambdau^nt^n}
         {\displaystyle \sum_{n=0}^{+\infty}b_n\lambdau^nt^n}.
\end{equation}
This gives, after equating terms of the same power of $t$, the following system
\begin{subequations}\label{eq:equation-system-2d}
\begin{eqnarray}
-\lambdau^na_n+b_n &=& 0 \\
-\delta a_n+\left(c+9b_0^2-\lambdau^n\right)b_n &=& s_{n-1},
\end{eqnarray}
\end{subequations}
where $s_{n-1}$ is defined by
\begin{equation}
s_{n-1}:=-3\left(\sum_{j=1}^{n-1}b_0b_{n-j}b_j
                +\sum_{i=1}^{n-1}\sum_{j=0}^{i}b_{n-i}b_{i-j}b_j\right).
\end{equation}
The above is a linear system of equations for the coefficients  $a_n, b_n$ of the power series, whose zero-th
order terms ($a_0, b_0$) are both zero since they represent the coordinates of the fixed point at the origin. The first order terms ($a_1, b_1$) on the other hand are simply the coordinates of the unstable eigenvector.
Thus, we are now ready to compute the constants $a_n, b_n$ for every, finite, value of $n>1$.

To perform the numerical computation of $\Su(t)$ we first truncate the series up to a polynomial
$\Pu(t)$ of degree $N$ which is evaluated using Horner's method. Still this polynomial will only be a
good approximation of the unstable manifold for a restricted range of values of $t$.
The range of validity can be quantified as follows:

\begin{definition} \label{def:epsilon-radius}
  Let $\varepsilon>0$. Define
  $\tau>0$ to be an $\varepsilon$-radius of validity of the polynomial
  approximation $\Pu(t)$ of $\Su(t)$ if $\max_{t\in [-\tau,\tau]} \arrowvert
  \arrowvert f\circ \Pu(t)-\Pu(\lambdau t)\arrowvert \arrowvert <\varepsilon$.
\end{definition}

\begin{figure}
\begin{center}
\includegraphics[scale=1.2]{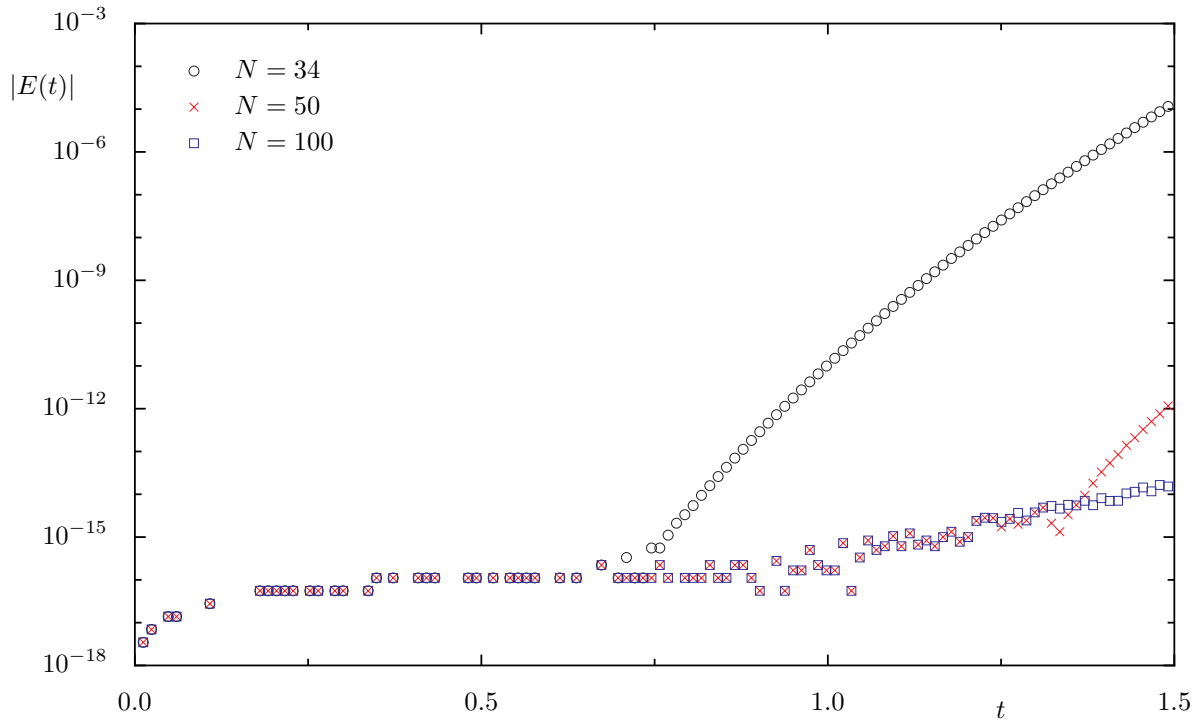}
\end{center}
\caption{\label{fig:epsilon-radius}%
Plot of the error $E(t)=\arrowvert \arrowvert f\circ \Pu(t)-\Pu(\lambdau t)\arrowvert \arrowvert$ for different degrees $N$ of the polynomial for $\delta=1.0$. This shows that with 100 coefficients and $t$ up to
1.5 the error is below $2\cdot 10^{-14}$.}
\end{figure}

Using this prescription, we shall keep terms up to order $N=34$ of our series, for which we numerically obtain that $\tau = 0.75$ provides an $\varepsilon$-radius of validity, when we choose $\varepsilon=10^{-15}$, see figure~\ref{fig:epsilon-radius}. For $t>\tau$ the error strongly increases.
As also shown in this figure, increasing the order to $N=100$ we can go as far as $\tau=1.5$ for  $\varepsilon=2\cdot 10^{-14}$, which turns out to be sufficient for accurately locating the primary (and all subsequent) homoclinic intersections, as explained below. Note, that in the definition
of the error determining the $\varepsilon$-radius of validity,
the difference between points at $\lambdau t$ is considered.
Thus, for practical purposes we may assume that the parametrization method provides satisfactory results up to $\lambdau \tau$.

To obtain an equally good polynomial approximation $\Ps(t)$ for the stable manifold $\Ws(0, 0)$ of the origin, we repeat the above steps, replacing $\lambdau$ with the stable eigenvalue $\lambdas$ in equation~\eqref{eq:defining-equation-stable-manifold}, and proceed to solve for the coefficients of the parametrization of the manifold. First, we replace the coefficients of the first order terms with the coordinates of the stable eigenvector. Actually, in the $\delta = 1$ case one may exploit the fact that $f$ is differentially conjugate to its inverse under $\rho(x,y)=(y,x)$, and obtain its stable manifold as the image of its unstable manifold under $\rho$.

This symmetry is directly reflected in the coefficients of the polynomials of the parametrization method:
\begin{lemma}\label{lemma:symm-coeffs}
  Let $\Su(t)=\left(\sum_{n=0}^{+\infty}a_nt^n, \sum_{n=0}^{+\infty}b_nt^n\right)$
  be a parametrization of the unstable
  manifold at the origin of the symplectomorphism $f$ above. Then, a parametrization
  of the stable manifold of $f$ at the origin is
  $\Ss(t)=\left(\sum_{n=0}^{+\infty}b_nt^n,\sum_{n=0}^{+\infty}a_nt^n\right)$.
\end{lemma}

\begin{proof}
Following from the system of equations~\eqref{eq:equation-system-2d}
with $\delta=1$ and $b_0=0$, the coefficients $a_n,\ b_n$ of $\Su(t)$
satisfy the following system of equations
 \begin{equation}
 \label{mancon1}
  -\lambdau^na_n+b_n=0\\
 \end{equation}
 \begin{equation}
 \label{mancon11}
  -a_n+(c-\lambdau^n)b_n=-3\sum_{i=1}^{n-1}\sum_{j=0}^{i}b_{n-i}b_{i-j}b_j,
 \end{equation}
Let us now suppose that
$\Ss(t)=(\sum_{n=0}^{+\infty}C_nt^n,\sum_{n=0}^{+\infty}D_nt^n)$ represents  the parametrization of the
stable manifold of the origin. Note that, since $f$ is a symplectomorphism, the stable eigenvalue
$\lambdas$ equals $\lambdau^{-1}$. Clearly, the coefficients $C_n,\ D_n$ satisfy the following system:
\begin{equation}
\label{mancon2}
  -\lambdas^nC_n+D_n=0\\
\end{equation}
\begin{equation}
\label{mancon21}
  -C_n+(c-\lambdas^n)D_n=-3\sum_{i=1}^{n-1}\sum_{j=0}^{i}D_{n-i}D_{i-j}D_j.
\end{equation}

We claim that $C_n=b_n,\ D_n=a_n$. Indeed it follows from (\ref{mancon2}) that:
\begin{equation}
-\frac{1}{\lambdau^n}C_n+D_n=0
\qquad \Leftrightarrow\qquad
 -\frac{1}{\lambdau^n}b_n+a_n=0,
\end{equation}
which is equivalent with (\ref{mancon1}), while from (\ref{mancon21}) we get
\begin{subequations}
\begin{alignat}{5}
 && -C_n+(c-\lambdas^n)D_n &\;\;=\;\;&& -3\sum_{i=1}^{n-1}\sum_{j=0}^{i}D_{n-i}D_{i-j}D_j\\
\Leftrightarrow \qquad &&
   -C_n+(c-\frac{1}{\lambdau^n})D_n &\;\;=\;\;&& -3\sum_{i=1}^{n-1}\sum_{j=0}^{i}D_{n-i}D_{i-j}D_j\\
 \Leftrightarrow \qquad && -b_n+(c-\frac{1}{\lambdau^n})a_n &\;\;=\;\;&&-3\sum_{i=1}^{n-1}\sum_{j=0}^{i}a_{n-i}a_{i-j}a_j\\
\Leftrightarrow \qquad && -b_n+(c-\frac{1}{\lambdau^n})a_n &\;\;=\;\;&& -3\sum_{i=1}^{n-1}\sum_{j=0}^{i}\frac{1}{\lambdau^{n-i}}b_{n-i}\frac{1}{\lambdau^{i-j}}b_{i-j}\frac{1}{\lambdau^{j}}b_j\\
 \Leftrightarrow \qquad && -a_n+(c-\lambdau^n)b_n &\;\;=\;\;&& -3\sum_{i=1}^{n-1}\sum_{j=0}^{i}b_{n-i}b_{i-j}b_j,
\end{alignat}
\end{subequations}
leading us back to (\ref{mancon11}). Thus, $C_n=b_n$ and $D_n=a_n$ as claimed.
\end{proof}

So, either by solving the defining equation, or making use of
\ref{lemma:symm-coeffs}, we arrive at
a polynomial $\Ps(t)$ of degree $N=34$, which is a satisfactory
approximation of the stable manifold $\Ss(t)$, in the sense of \ref{def:epsilon-radius}
and with the same radius of validity.

Now, since the polynomial curve $\Pu:[-\tau, \tau] \rightarrow \mathbb{R}^2$
provides a good approximation of the local unstable
manifold of the origin, we may iterate it using the mapping
$f$ to produce an approximation of the unstable manifold.
The same holds for $\Ps(t)$, from which
the corresponding stable manifold $\Ws$ can be obtained
by iteratively applying $f^{-1}$.

\begin{figure}[t]
\begin{center}
\includegraphics[scale=1.2]{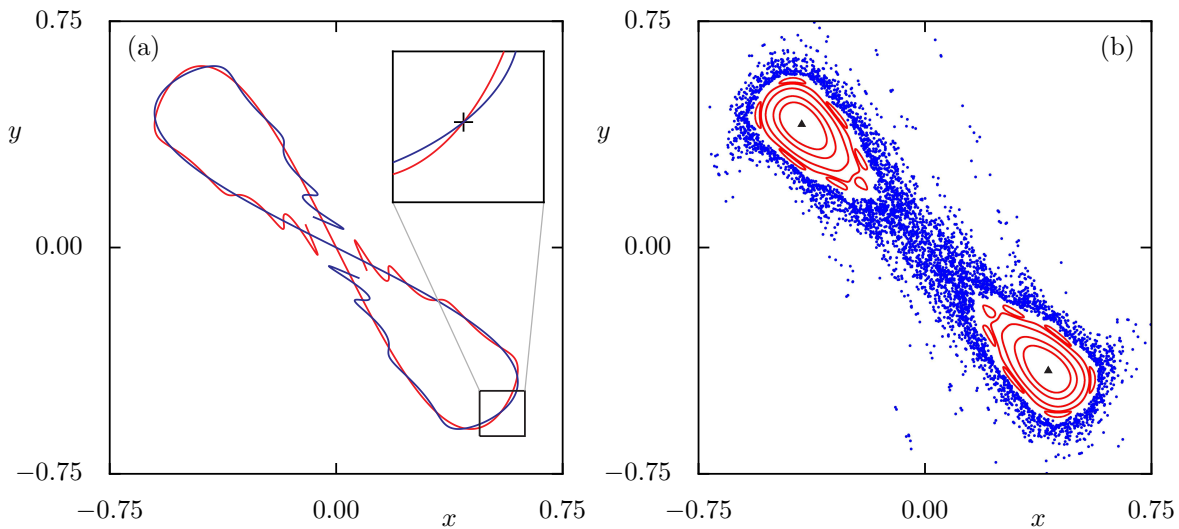}
\end{center}
 \caption{\label{2dintersectiond1}%
          (a) Stable (blue) and unstable (red) manifolds computed using
          the parametrization method for $\delta=1$.
          The inset shows a magnification
          around the primary homoclinic intersection
          at $(\xh, \yh) \approx (0.54527107, -0.54527107)$
          which is indicated by a cross.
          (b) Regular orbits (red curves) around the periodic points of
          period 2 (black triangle) and several irregular orbits (blue dots)
          escaping
          via the homoclinic tangle.}
\end{figure}

Figure \ref{2dintersectiond1}(a) shows the stable and unstable manifolds of the origin for $\delta=1$. To obtain this plot, the parametrization method
was used up to $t=t_{\text{max}}$ with $t_{\text{max}}=0.42$.
For example, for the unstable manifold, the segment obtained for
$t\in [t_{\text{max}}/\lambdau, t_{\text{max}}]$ is iterated up to $6$ times
using $f$ and then the corresponding segment of the stable manifold was obtained by symmetry. Note that, as the fixed point is inverse hyperbolic, one obtains
alternatingly segments lying in the second and fourth quadrant.

Since a homoclinic point of transverse intersection between the stable and unstable
manifolds exists, the Birkhoff-Smale theorem can be invoked to guarantee the
existence of an infinity of transverse intersections, a phenomenon known as
homoclinic chaos \cite{wiggins1}.
This explains the scattered points shown in the phase space of the mapping
at $\delta=1$, plotted in figure~\ref{2dintersectiond1}(b),
where several orbits started in the region close to the
hyperbolic fixed point at the origin eventually escape to infinity after
a few iterations via the homoclinic tangle. Some elliptic periodic orbits as well as tori encircling the stable period--$2$ points in the 2nd and 4th quadrant are also shown.

\subsection{Homoclinic intersections}
\label{homoclinicity}
We now wish to compute an approximation of the homoclinic intersection, i.e.\ the point
$(\xh, \yh)$ at which $\Ws(0,0)$ and $\Wu(0,0)$ intersect transversely.
This means, that there are natural numbers $\numu, \nums$, along with $\tu, \ts\in \mathbb{R}$ such that
$f^{\numu}(\Pu(\tu))=f^{-\nums}(\Ps(\ts))=(\xh, \yh)$.
Furthermore, since the intersection is transversal, the vectors
$\frac{\partial}{\partial \tu}f^{\numu}(\Pu(\tu))$ and
$\frac{\partial}{\partial \ts}f^{-\nums}(\Ps(\ts))$ are independent.

To evaluate such homoclinic points, we define the map $\Phi:\mathbb{R}\times\mathbb{R}\rightarrow \mathbb{R}^2$, as $\Phi(\tu,\ts)=f^{\numu}(\Pu(\tu))-f^{-\nums}(\Ps(\ts))$ and search for its zeros. One solution, of course, is the point  $(\tu,\ts)=(0,0)$ corresponding to the origin. Non--trivial roots $(\tu,\ts)$ of this map for $(\numu,\nums)\in \mathbb{N}\times \mathbb{N}$
correspond to transverse homoclinic points of the mapping $f$. Validated computations can now be used to analytically prove the existence of such solutions for fixed parameter values, as described in \cite{jdm, HarCanFigLuqMon2016, Mir2015b}. Here however, we are more interested in accurately obtaining these solutions, while permitting the parameters to vary, until they no longer exist. Note that this is also the subject of reference \cite{Capi-Zgli}, where the authors mention the possibility of applying the parametrization method, along with their technique, to obtain computer--assisted proofs of
transversal intersections of manifolds.
Here we proceed as described below.

To determine the homoclinic intersections for a sequence of decreasing values of $\delta$ it is convenient to start from an already computed homoclinic intersection at one $\delta$ value and use the corresponding $(\tu,\ts)$ as a starting point for finding the solution of $\Phi(\tu,\ts)= (0, 0)$ for a slightly smaller $\delta$. If no solution is found, the step size is reduced, so that effectively a bisection in $\delta$ is performed, approximating the critical value $\deltac$ for which no homoclinics exist.

To determine the non--trivial roots of $\Phi$
it turns out to be particularly convenient to use polynomials of degree $N=100$
as, under these circumstances, the range for $\tu, \ts$ can be extended to $[-1.6, 1.6]$
while the manifolds are still computed with sufficient accuracy,
namely with $|E(t)| < 4 \cdot 10^{-14}$.
Thus, we do not have to worry about the possibility of non-intersecting segments, noting also that segments of each manifold get mapped to the other quadrant across $(0,0)$ due to the inverse hyperbolicity of the fixed point. For $\delta=1.0$, for example, we obtain a zero of $\Phi(\tu,\ts)$ for $(\tu,\ts) \approx (1.5849, -1.5849)$ and $\numu=\nums=0$. This means that the manifolds can be represented up to the homoclinic point with high enough accuracy, using the parametrization method, while no additional application of the mapping $f$ or its inverse is necessary. The corresponding homoclinic point is located at $(\xh,\yh)=(0.545271067753899,-0.545271067753900)$.
The fact that both $\tu + \ts = 0$ and $\xh + \yh=0$ within our numerical accuracy is already a good test of the quality of the numerically determined homoclinic point, particularly because the symmetry of the mapping for $\delta=1.0$ has not been used in the computations.

\begin{figure}[t]
\begin{center}
\includegraphics[scale=1.2]{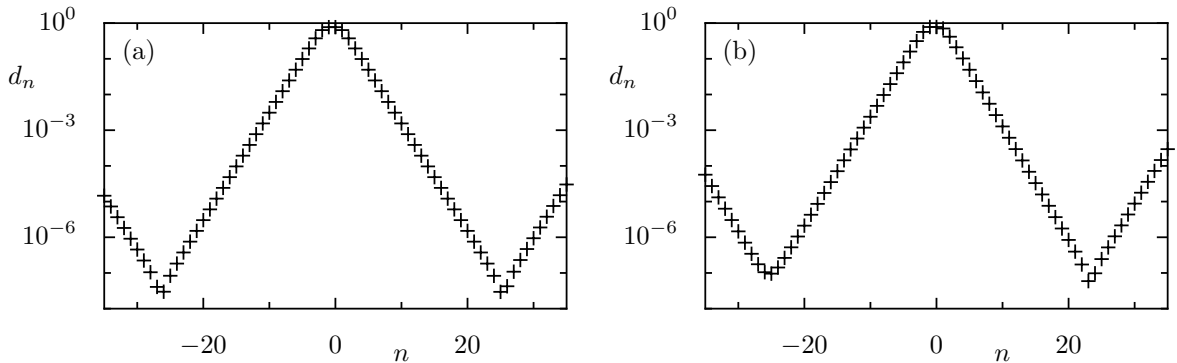}
\end{center}
\caption{\label{fig:homoclinic-points-accurarcy-2d}
Plot of the distance $d_n$ of the iterates $f^n(\xh, \yh)$ from the origin. Note that the iterates of our computed homoclinic point $(\xh, \yh)$ pass by the origin but stay away from it by an amount that provides a measure of the inaccuracy of the calculation. (a) $\delta=1.0$ and (b) $\delta=0.9714375$.
}
\end{figure}

To test the accuracy of the computed homoclinic intersection $(\xh, \yh)$ we determine the distance $d_n$ of $f^n(\xh, \yh)$ from the origin
for both positive and negative $n$. For positive $n$ the iterates of the homoclinic point
approach the origin along the stable manifold $\Ws$, while, for negative $n$ the approach is along the unstable manifold $\Wu$.
Any inaccuracy of the determined homoclinic intersection point (including additional round-off errors when applying $f^n$) implies that $d_n$ eventually will increase again. This means that the inaccurate orbit will depart again from the origin along $\Wu$ for positive $n$ and along $\Ws$ for negative $n$. The smallest distance to the origin thus gives a
measure of the inaccuracy of the numerically determined homoclinic intersections,
see figure~\ref{fig:homoclinic-points-accurarcy-2d}(a) for $\delta=1.0$ and (b) for $\delta<1$.

So far we have relied on a graphical verification that the homoclinic intersection is actually transversal, i.e.\ that the vectors $\frac{\partial}{\partial \tu}f^{\numu}(\Pu(\tu))$ and $\frac{\partial}{\partial \ts}f^{-\nums}(\Ps(\ts))$ are independent. Due to the parametrization this can be easily checked by computing the determinant of the vectors $Df^{\numu} \frac{\partial}{\partial \tu} (\Pu(\tu))$ and $Df^{-\nums} \frac{\partial}{\partial \ts} (\Ps(\ts))$ for $\numu=\nums=0$. This is shown in figure~\ref{fig:determinant}. With decreasing $\delta$ the determinant becomes smaller which means that the area spanned by the two tangent vectors becomes smaller and smaller until it becomes zeros at $\deltac$. The plot at figure~\ref{fig:determinant} behaves like $\sqrt{\delta -\deltac}$. This suggests that at the critical parameter value $\deltac$ the tangency of the stable and unstable manifolds is quadratic, which is a manifestation of the genericity of their intersection.
A fit to $a \sqrt{\delta -\deltac}$ gives
for the critical parameter $\deltac=0.9713965579$.

\begin{figure}[t]
\begin{center}
\includegraphics[scale=1.2]{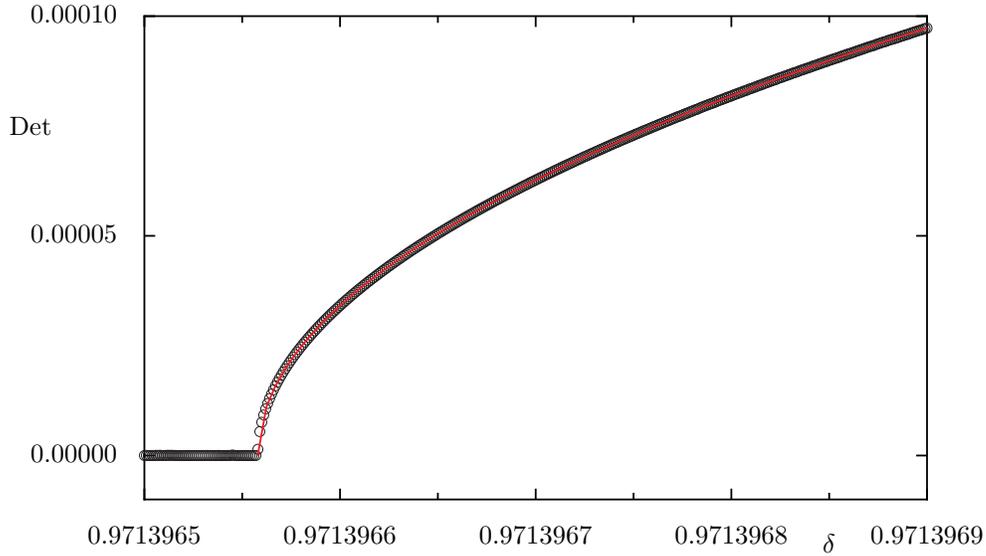}
\caption{\label{fig:determinant}
Plot of the determinant of the two vectors tangent to the stable and unstable manifolds at the homoclinic point, as a function of $\delta$. The determinant is zero below the critical parameter value $\deltac$, while it tends to zero in a manner implying a quadratic tangency at this point.
The red line is a fit to $a \sqrt{\delta -\deltac}$.
}
\end{center}
\end{figure}

\subsection{Homoclinic tangency: the $\delta \simeq 0.971397$ case}

Let us now follow the primary homoclinic intersection of the two manifolds as we decrease
$\delta$.  Numerically we find a solution of
$\Phi(\tu,\ts)=(0, 0)$ as long as
$\delta \in [\deltac, 1]$ with $\deltac=0.971397$, meaning that below this $\deltac$ homoclinic orbits no longer exist. Indeed, for $\delta<\deltac$ we do not obtain zeros of $\Phi(\tu,\ts)$ with the required accuracy, i.e.\ components of $|\Phi(\tu,\ts)|$ become larger than $10^{-15}$. Figure~\ref{fig:homoclinic-points-accurarcy-2d}(b)
shows the corresponding distances $d_n$ for $\delta=\deltac$ and demonstrates that the numerically determined homoclinic intersection is of comparable accuracy as in the case $\delta=1.0$.

For the above mentioned value of $\delta=\deltac$ we compute the invariant
manifolds using the parametrization method
(followed by iterating the local manifolds).
The result is shown in figure \ref{fig:2dtangency}(a)
and visually confirms that this is approximately the
parameter at which the manifolds become tangent.
For even smaller $\delta=0.96$ one clearly
finds that the stable and unstable manifold no longer
intersect, as seen in figure \ref{fig:2dtangency}(b).

\begin{figure}[b]
\begin{center}
\includegraphics[scale=1.2]{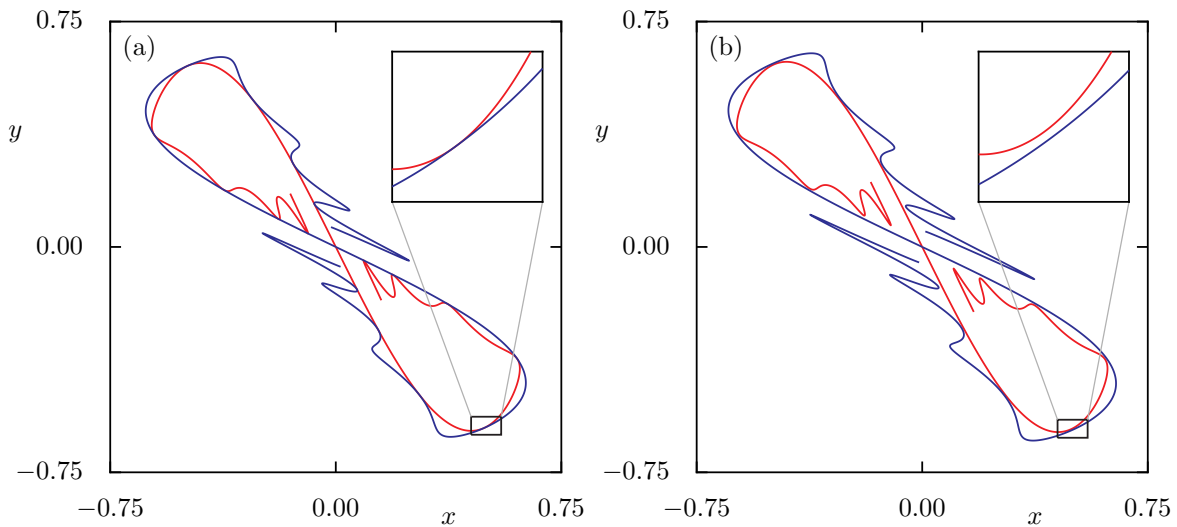}
\end{center}
\caption{\label{fig:2dtangency}%
(a) Tangency of the stable and unstable manifolds shown at $\delta=0.971397$.
(b) For $\delta=0.96$ the stable and unstable manifolds no longer
intersect.}
\end{figure}

We now turn our attention to linearly coupled 2--D mappings of the type studied above. Our
objective is to calculate efficiently and accurately primary homoclinic orbits in the phase space of 4--D maps approximating breathers in coupled Hamiltonian chains. We wish to find out how computationally demanding is this task, and compare the results with analogous findings in the 2--D case.

\section{Two coupled cubic systems}
\label{4-d case}

To investigate breather interactions in two linearly coupled chains of
1--dimensional Hamiltonian systems, the approach discussed in the introduction
leads to the identification of homoclinic orbits of a 4--D map of the form:
\begin{subequations}
\begin{eqnarray}
A_{n+1}-cA_n+\delta A_{n-1}=3A^3_n+b(A_n-B_n)\\
B_{n+1}-cB_n+\delta B_{n-1}=3B^3_n-b(A_n-B_n).
\end{eqnarray}
\end{subequations}
where $A_n$ and $B_n$ are the leading terms of the Fourier coefficients corresponding to the discrete breathers in the first and second chain respectively. Setting $A_{n-1}=x_1,\ A_n=y_1,\ B_{n-1}=x_2,\ B_n=y_2$,
we define the mapping $f:\mathbb{R}^4 \rightarrow \mathbb{R}^4$
\begin{equation} \label{eq:our-4d-map}
f(x_1,y_1,x_2,y_2)
  =\FOURVEC{y_1}
            {cy_1-\delta x_1+3y_1^3+b(y_1-y_2)}
            {y_2}
            {cy_2-\delta x_2+3y_2^3-b(y_1-y_2)},
\end{equation}
which is a diffeomorphism with inverse
$f^{-1}:\mathbb{R}^4\rightarrow \mathbb{R}^4$ given by
\begin{equation}
  f^{-1}(x_1,y_1,x_2,y_2) =
    \FOURVEC{\frac{1}{\delta}((c+b)x_1+3x^3_1-bx_2-y_1)}
            {x_1}
            {\frac{1}{\delta}((c+b)x_2+3x^3_2-bx_1-y_2)}
            {x_2}.
\end{equation}
This map possesses a number of fixed points. However, we are only
interested in the origin, which must be a saddle fixed point as in the 2--dimensional case. Thus, we choose parameter values for which the origin of the above 4--D map possesses a 2--dimensional stable manifold, and a 2--dimensional unstable manifold, corresponding to pairs of real eigenvalues with $|\lambda_{1,2}|>1$ and $|\lambda_{3,4}|<1$ respectively.

Let us suppose that $\Su:\mathbb{R}^2\rightarrow \mathbb{R}^4$ is the parametrization of the unstable manifold of the origin, corresponding to the eigenvalues $\lambda_1,\ \lambda_2$.
If
\begin{equation}
\Su(u,v)=\left(\sum_{n=0}^{+\infty}\sum_{m=0}^{+\infty}a_1^{nm}u^nv^m,\sum_{n=0}^{+\infty}\sum_{m=0}^{+\infty}a_2^{nm}u^nv^m,\sum_{n=0}^{+\infty}\sum_{m=0}^{+\infty}a_3^{nm}u^nv^m,\sum_{n=0}^{+\infty}\sum_{m=0}^{+\infty}a_4^{nm}u^nv^m\right)
\end{equation}
represents its power--series expansion, where $a_i^{nm},\ i\in\{1,2,3,4\}$,
are coefficients of the mononomials $u^nv^m$, the defining equation of the manifold becomes
\begin{equation}
 f\circ \Su(u,v)=\Su(\lambda_1 u,\lambda_2 v).
\end{equation}
The left--hand side of this equation reads
\begin{equation*}
 \FOURVEC{\displaystyle \sum_{n,m}a_2^{nm}u^nv^m}
         {\displaystyle (c+b)\sum_{n,m}a_2^{nm}u^nv^m-\delta \sum_{n,m} a_1^{nm}u^nv^m+3\left(\sum_{n,m}a_2^{nm}u^nv^m\right)^3-b \sum_{n,m}a_4^{nm}u^nv^m}
         {\displaystyle \sum_{n,m}a_4^{nm}u^nv^m}
         {\displaystyle (c+b)\sum_{n,m}a_4^{nm}u^nv^m-\delta \sum_{n.m}a_3^{nm}u^nv^m+3\left(\sum_{n,m}a_4^{nm}u^nv^m\right)^3-b \sum_{n,m}a_2^{nm}u^nv^m}
\end{equation*}
while the right--hand side equals
\begin{equation*}
 \FOURVEC{\displaystyle \sum_{n,m}a_1^{nm}\lambda_1^n \lambda_2^m u^nv^m}
         {\displaystyle \sum_{n,m}a_2^{nm}\lambda_1^n \lambda_2^m u^nv^m}
         {\displaystyle \sum_{n,m}a_3^{nm}\lambda_1^n \lambda_2^m u^nv^m}
         {\displaystyle \sum_{n,m}a_4^{nm}\lambda_1^n \lambda_2^m u^nv^m},
\end{equation*}
where by $\sum_{n,m}$ we denote double summation over $n$ and $m$ ranging from 0 to $\infty$.

Equating now, as before, terms of the same degree, we arrive at
the following system of equations for the coefficients
\begin{subequations}
\begin{alignat}{4}
 -\lambda_1^n \lambda_2^m a_1^{nm}+a_2^{nm} &=&& \,\,0, \\
 -\delta a_1^{nm}+(c+b-\lambda_1^n \lambda_2^m)a_2^{nm}-ba_4^{nm} &=&& -3 \sum_{k=0}^{n}\sum_{l=0}^{m}\sum_{i=0}^{k}\sum_{j=0}^{l}a_2^{n-k,m-l}a_2^{k-i,l-j}a_2^{i,j},\\
 -\lambda_1^n \lambda_2^m a_3^{nm}+a_4^{nm} &=&& \,\, 0, \\
 -ba_2^{nm}-\delta a_3^{nm}+(c+b-\lambda_1^n \lambda_2^m)a_4^{nm} &=&& -3 \sum_{k=0}^{n}\sum_{l=0}^{m}\sum_{i=0}^{k}\sum_{j=0}^{l}a_4^{n-k,m-l}a_4^{k-i,l-j}a_4^{i,j}.
\end{alignat}
\end{subequations}

Note that the sums on the right side contain terms of the form
$a_i^{nm},\ i\in\{2,4\}$, as well. Although these terms should
be isolated and transferred to the left side of the equations, we prefer to write them more ``compactly'' in the above form.

Cumbersome as they may seem, these equations are again linear with respect to the
unknown $a_i^{nm}$, and may be solved immediately provided the terms of
lower degree are known. The above equations will be crucially used in what follows to
locate homoclinic points of the 4--D map $f$.

\subsection{Error estimates for the parametrization of the manifolds}

To study the 4--D map \eqref{eq:our-4d-map} we shall start from the uncoupled case $b=0,\ \delta=1$, and continue with the coupled case for small positive values of $b$. In this study, we set $b=0.1$ and continue the homoclinic intersections for $\delta<1$ until they completely disappear. The question is whether these intersections persist as robustly as they do in the 2--dimensional case. To find out, we shall first consider an error estimate for the polynomial approximation of the series representing the (un)stable manifold.

Let us, therefore, define an analogous error function as in the 2--dimensional case
\begin{equation}
 E:[-2,2]\times [-2,2]\rightarrow \mathbb{R},\
        E(u,v)=\lVert f \circ \Pu(u,v)-\Pu(\lambda _1u,\lambda _2v)\rVert,
\end{equation}
where $\Pu$ stands for the polynomial approximation of the unstable manifold,
consider polar coordinates $(u,v)=(r\cos\theta,r\sin\theta)$ and plot $\lvert E(r,\theta)\lvert$ as a function of $r\in [0,2]$, for various values of $\theta$,
see figure~\ref{fig:err-4d}.

\begin{figure}
\begin{center}
\includegraphics[scale=1.2]{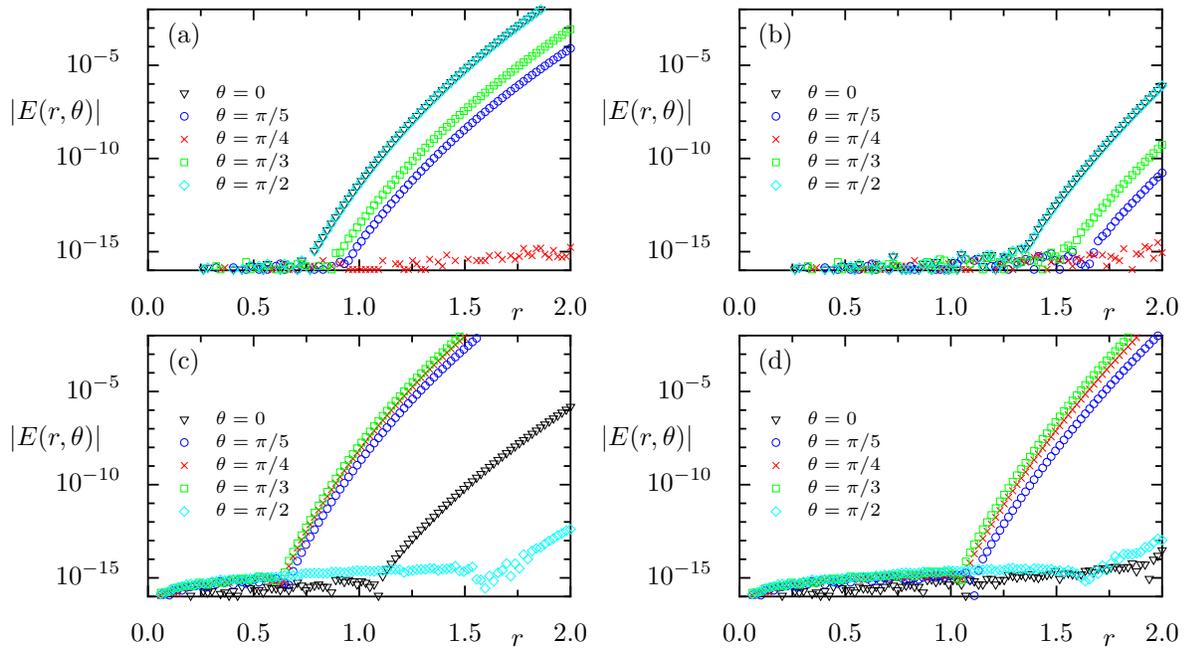}
\end{center}
\caption{\label{fig:err-4d}
Evaluation of the error function $E(r,\theta)$ vs. $r$ for the parametrization method for different values of $\theta$ and (a) $b=0.0$, $N=34$, (b) $b=0.0$, $N=50$, (c) $b=0.1$, $N=34$, and (d) $b=0.1$, $N=50$. Note how the accuracy of $10^{-15}$ extends to longer segments of the manifolds as $N$ increases.
        }
\end{figure}

We first observe that, for all $b\in [0,0.1]$, keeping terms up to order 50 in the polynomial representation $\Pu$ of the unstable manifold yields a region of validity with $r=1$ corresponding to an error magnitude of order $10^{-15}$. An analogous statement
holds for the stable manifold as well. We thus keep the interval $r\in [0,1]$ as the domain of definition of our approximations $\Pu(\uu,\vu)$ and $\Ps(\us,\vs)$. Our aim is to locate non-trivial zeros of the mapping
\begin{equation}
 \Phi (\uu,\vu,\us,\vs)=f^{\numu}(\Pu(\uu,\vu))-f^{-\nums}(\Ps(\us,\vs)).
\end{equation}
Note, that as we will be using a sufficiently high order for the polynomials
$\Pu$ and $\Ps$, the homoclinic intersections already occur
for values of $\uu, \vu, \us, \vs < 1.16$ when using $\numu = \nums = 0$.
As in the 2--D case we assume that the parametrization method provides
accurate results for this slightly extended range.

\subsection{Homoclinic points for $b=0$}

For $b=0$ our system reduces to two uncoupled 2--D maps of the form studied in previous sections. Thus the geometry in phase space is given by the direct product of two independent maps in $(x_1, y_1)$ and $(x_2, y_2)$. In each of these 2--D maps the origin is a fixed point and $(\xe, \ye) = (\pm 1/\sqrt{6}, \mp 1/\sqrt{6})$ are the coordinates of an elliptic periodic point of period $2$.
This implies that for $b=0$ the 4--D map has: (i) The origin $(x_1, y_1, x_2, y_2) = (0, 0, 0, 0)$ as a hyperbolic-hyperbolic fixed point, (ii) two elliptic-hyperbolic period--$2$ orbits at $(x_1, y_1, x_2, y_2) = (\pm \xe, \pm \ye, 0, 0)$ and $(x_1, y_1, x_2, y_2) = (0, 0, \pm \xe, \pm \ye)$, (iii) two elliptic-elliptic period--$2$ orbits, one at $(\pm \xe, \pm \ye, \pm \xe, \pm \ye)$ and another at $(\pm \xe, \pm \ye, \mp \ye, \mp \ye)$.

\begin{figure}[b]
\begin{center}
\includegraphics[scale=1.2]{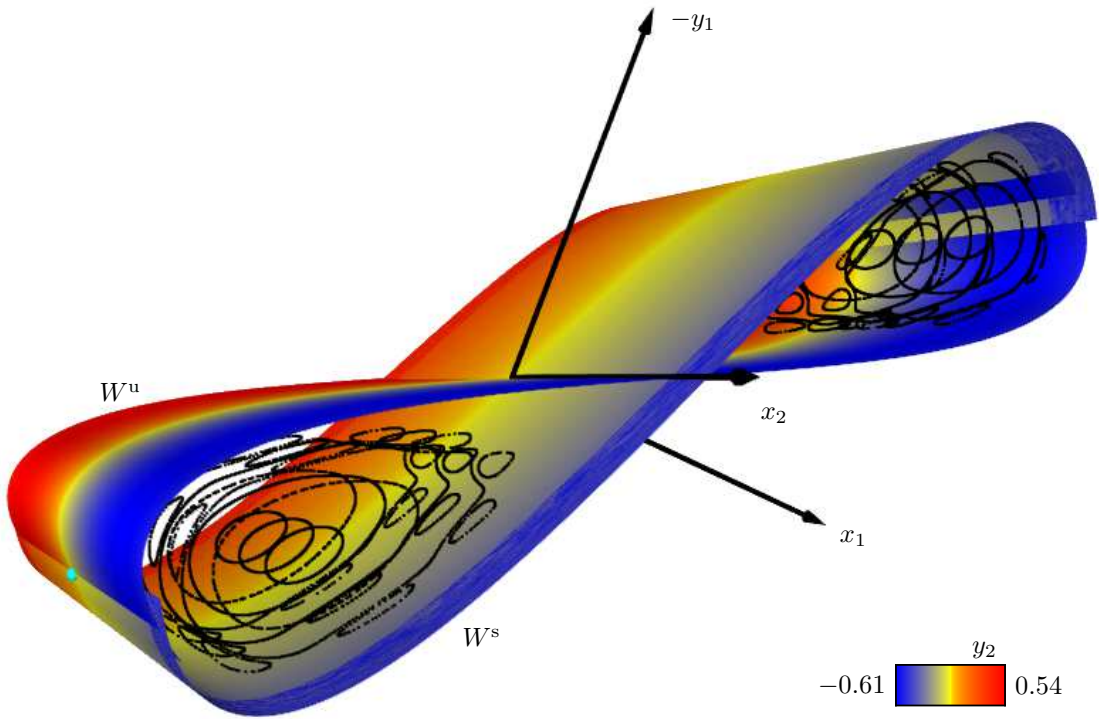}
\end{center}
\caption{\label{fig:uncoupled-intersection}%
Homoclinic intersection shown by a small cyan sphere on the left part of the figure in the uncoupled case $b=0$. Shown are the 2--dimensional stable and unstable manifolds
as projections onto $(x_1, x_2, y_1)$ space, with $y_2$ encoded in color. The black dots
show some regular tori in a 3--dimensional phase space slice representation.
\movierefall}
\end{figure}

In order to visualize the regular dynamics occurring in the 4--dimensional phase space
we use a 3--dimensional phase space slice \cite{RicLanBaeKet2014}.
As a convenient condition for determining our slice we set $y_2^\star = \ye$, since it includes the domain surrounding one point of each elliptic-elliptic orbit of period 2.
Whenever a point $(x_1, y_1, x_2, y_2)$ of a trajectory
fulfills the slice condition $|y_2 - y_2^\star| < 10^{-4}$
the remaining coordinates $(x_1, y_1, x_2)$ are shown in a 3--D plot.

Figure~\ref{fig:uncoupled-intersection} shows for the uncoupled case
several regular orbits (black dots) surrounding the points of the elliptic-elliptic orbit of period 2. The 2--dimensional manifolds $\Ws$ and $\Wu$
have been computed using the parametrization method. They are embedded in the 4--dimensional phase-space and shown as projections onto $(x_1, x_2, y_1)$. The projected coordinate $y_2$ is encoded in color. Thus a homoclinic intersection
of  $\Ws$ and $\Wu$ occurs at an intersection of the two surfaces, if and only if at the same point the color, i.e.\ the $y_2$ coordinate, also agrees. This point is indicated by a small cyan sphere on the left side of figure~\ref{fig:uncoupled-intersection}. In this case, the coordinates of the homoclinic points are given by $(\xh, \yh, 0, 0)$ and $(0, 0, \xh, \yh)$,
where $(\xh, \yh)$ are the coordinates of the homoclinic point
of the corresponding  2--D mapping studied in section~\ref{sec:2-dim-map}.

\subsection{Homoclinic points for $b>0$}

In order to determine the homoclinic intersection for different parameters $b$ and $\delta$ we start from the uncoupled case $b=0$, $\delta=1$ and follow the homoclinic intersections with increasing $b$ until $b=0.1$.
Numerically it is convenient to
use the homoclinic intersection, described by the
parameters $(\uu, \vu, \us, \vs)$, at one
value of $b$ as a starting point for finding the new root at a
slightly different value for $b$. Subsequently, after
we have found the homoclinic intersection at $b=0.1$ with $\delta=1$,
we fix $b=0.1$ and proceed to determine intersections at
lower values of $\delta$.

It turns out that the last intersection of the manifolds, at homoclinic tangency, occurs
at about $\delta = 0.99601$, which is considerably larger (closer to the conservative case) than in the single 2--D map. To test the accuracy of our computations we again consider the distance $d_n$ of $f^n(\xh_1, \yh_1, \xh_2, \yh_2)$ from the origin
for both positive and negative $n$, see figure~\ref{fig:4d-homoclinic-error}, for $\delta=0.997$. Interestingly enough, the magnitude of the observed minimal distance is comparable to the results for the 2--D case.

\begin{figure}
\begin{center}
\includegraphics[scale=1.2]{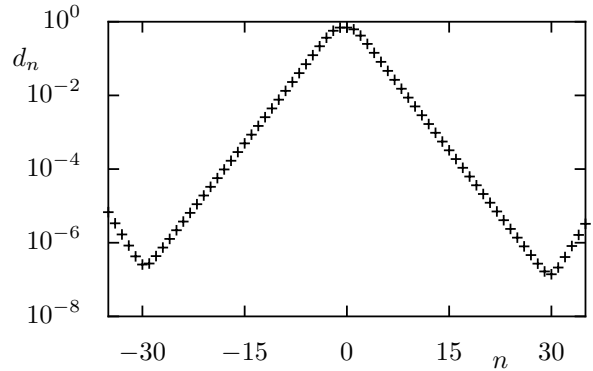}
\end{center}
\caption{\label{fig:4d-homoclinic-error}
Plot of the distance $d_n$ to the origin of iterates $f^n(\xh_1, \yh_1, \xh_2, \yh_2)$
of numerically determined  homoclinic intersection of the 4--D map for
$b=0.1$ and $\delta=0.997$.
}
\end{figure}

\begin{figure}[b]
\begin{center}
\includegraphics[scale=1.2]{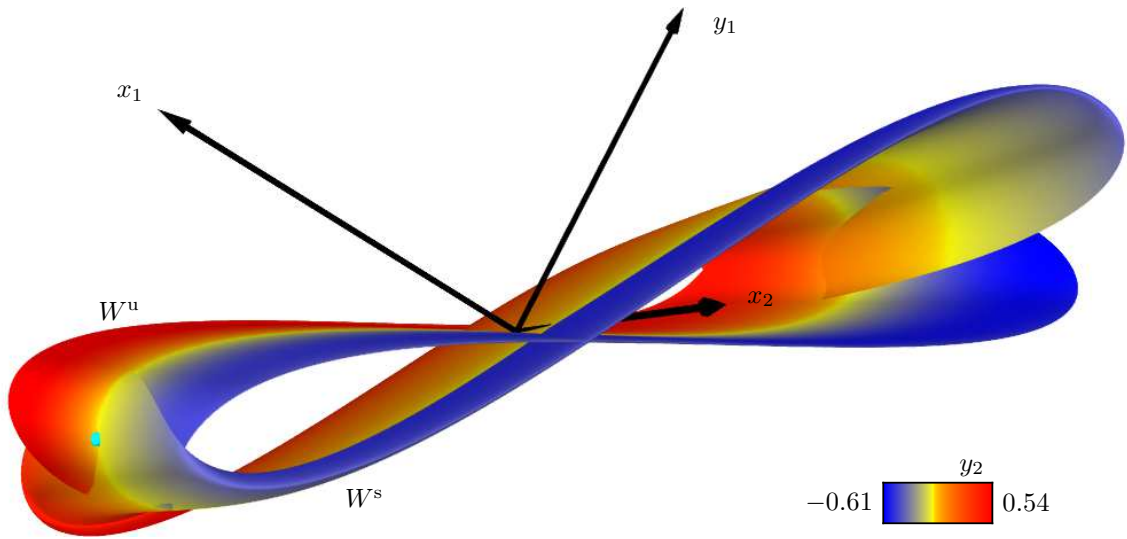}
\caption{\label{fig:coupled-intersection}
Stable and unstable manifolds computed using the parametrization
method for the 4--D map for $b=0.1$ and $\delta=0.997$.
A primary homoclinic intersection point is shown by a cyan sphere on the left side of the map.
\movierefall}
\end{center}
\end{figure}

\begin{figure}
\begin{center}
\includegraphics[scale=1.2]{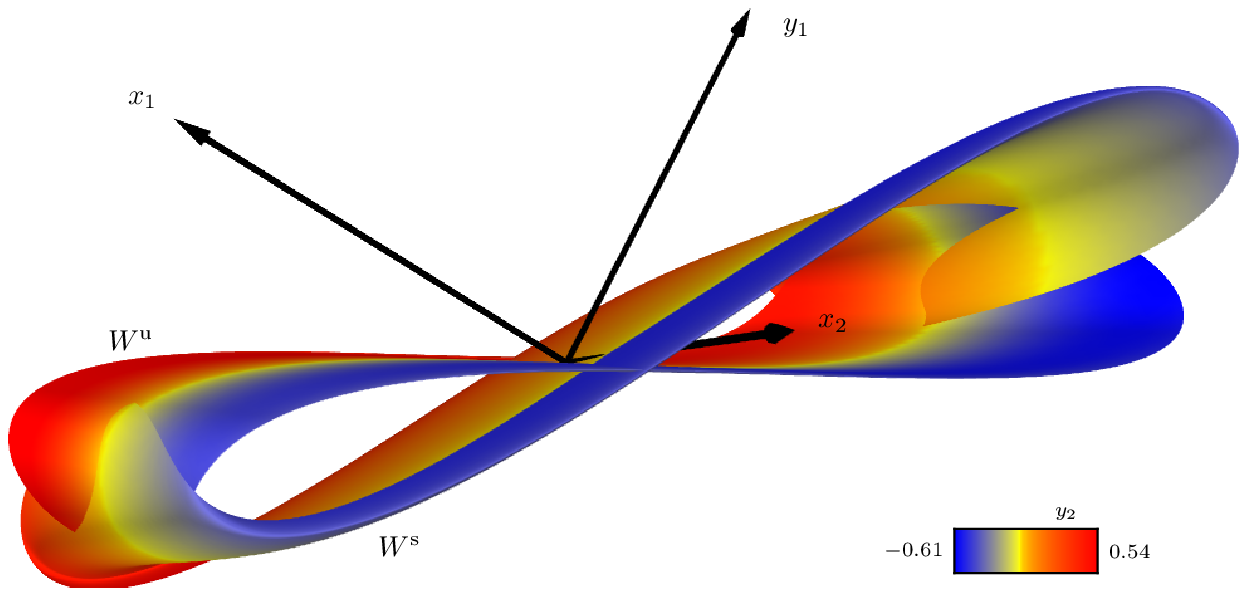}
\caption{\label{fig:coupled-no-intersection}
Stable and unstable manifolds computed using the parametrization
method for the 4--D map for $b=0.1$ and $\delta=0.99$.
No homoclinic intersection is found.
\movierefall}
\end{center}
\end{figure}

Analogously to the 2--D case, we calculated the determinant of the tangent vectors provided by the parametrization method at the homoclinic points, as a function of $\delta$. Similarly to figure~\ref{fig:determinant} the determinant approaches zero like $\sqrt{\delta -\deltac}$, which is a manifestation of the quadratic tangency of the invariant manifolds for $\delta = \deltac$.

Figure~\ref{fig:coupled-intersection}
shows a visualization of the two-dimensional manifolds $\Ws$ and $\Wu$
as projections onto $(x_1, x_2, y_1)$ with $y_2$ encoded in color.
The homoclinic intersection at
$(x_1, y_1, x_2, y_2) = (0.46521450, -0.49858860, -0.08725131, 0.08972831)$
corresponds to an intersection of the manifolds in $(x_1, x_2, y_1)$
and simultaneously in $y_2$, which means that the colors corresponding to the $y_2$ coordinate also agree. This point is indicated by a small cyan sphere at the left part of the figure.
While both the above conditions hold for the example in figure~\ref{fig:coupled-intersection},
the illustration for $b=0.1$ and $\delta = 0.99$ in figure~\ref{fig:coupled-no-intersection}
clearly shows that while the two manifolds intersect at several points in the projection onto $(x_1, x_2, y_1)$, the corresponding $y_2$ coordinates are different, as can be deduced from the different colors of the manifolds at their intersections in the projection.
Thus the homoclinic intersection ceases to exist very quickly as $\delta$ decreases from 1 for fixed $b=0.1$.

Note, that the parametrization method is not only numerically
convenient to compute the homoclinic intersections
of the 2--D manifolds, but also for their visualization.
In contrast to computing the manifolds based on the linearized dynamics
and application of the mapping,
a simple two-dimensional uniform grid in the parametrization
variables $(\uu, \vu)$ and $(\us, \vs)$, respectively,
leads to a well structured representation
of the 2--D manifolds embedded in the 4--D phase space
whose projection can be visualized in a straightforward way.
Thus, for obtaining the manifolds shown in
figures~\ref{fig:uncoupled-intersection},
\ref{fig:coupled-intersection}, and \ref{fig:coupled-no-intersection}
no further application of the map or its inverse is necessary.
Note, however, that this would be required if we were to show more lobes,
as for the used order of polynomials the approximation by
the parametrization method is only accurate until a little beyond the
homoclinic intersection.

\section{Conclusions}
\label{conclusions}
In this paper we studied 2--D and 4--D polynomial maps arising in the approximate evaluation
of the amplitudes of discrete breathers in single and double 1--dimensional Hamitonian lattices.
We began with a 2--D map, which belongs to the generalized H\'{e}non family,
and computed its invariant manifolds using convergent series expansions derived by the parametrization
method. We used parameter values at which tranversal intersections of the invariant manifolds of a
saddle point at the origin are expected to occur, and followed these intersections until they
disappeared by homoclinic tangency. Thus, we were able to verify independently the validity of
these results by direct iteration, showing that at the identified parameters the points we were
following cease to converge to the origin.

We then turned our attention to a 4--D map, obtained from two coupled 2--D maps
of the previous form, resulting from two linearly coupled 1--dimensional Hamiltonian
particle chains. Again we used the parametrization method to solve a system of linear equations
and obtain the invariant manifolds of the saddle at the origin. Using continuation methods, we
showed that despite the more complicated calculations, homoclinic orbits can
be located
without great difficulty and with comparable accuracy as in the 2--D map case. The validity of
our findings was checked again by finding parameter values at which homoclinic orbits disappear through
a tangency of the corresponding manifolds.
What is remarkable, both in the 2--D and in the 4--D case it
was possible to accurately locate homoclinic points
just using the parametrization method
without any additional application of the mapping or its inverse.

Regarding other potential applications of our techniques, it would be very interesting to use them
to approximate discrete breathers in coupled 1--dimensional Hamiltonian lattices, for which to our knowledge
no results are available. Furthermore, the potential use of this approach
to problems described by vector fields (and even p.d.e.'s) is worth
exploring, in view of the effectiveness of the parametrization method in
locating special solutions.

Moreover, higher--dimensional maps are interesting by themselves as they are known to arise in various physical applications like e.g.\ in the case of colliding particle beams of high energy accelerators, where the corresponding equations often appear in polynomial form (see e.g.\ \cite{VIB,BouSko2006}). Recently, we have started looking into a number of potential applications of our methods to physical problems such as those mentioned above and results are expected in future publications.

\section*{Acknowledgements}

One of us (TB) is grateful for discussions on higher dimensional maps with Prof.\ Haris Skokos and Prof.\ Michael Vrahatis, while SA acknowledges conversations on the parametrization method with Prof.\ Spyros Pnevmatikos. Furthermore, AB acknowledges support by the Deutsche Forschungsgemeinschaft under grant KE~537/6--1. The authors would like to thank the referees for a thorough reading of the paper and many valuable comments which helped us considerably improve our manuscript. All \threeD{} visualizations were created using
\textsc{Mayavi}~\cite{RamVar2011}.

~

\section*{References}

\end{document}